\documentclass[oneside,a4paper,11pt,notitlepage]{article}
\usepackage[left=0.8in, right=0.8in, top=1.5in, bottom=1.5in]{geometry}

\usepackage[T1]{fontenc} 
\usepackage[utf8]{inputenc} 
\usepackage[english]{babel}
\usepackage{amssymb}
\usepackage{lipsum} 
\usepackage{lmodern}
\usepackage{amsmath}
\usepackage{amsthm}
\usepackage{bm}
\usepackage{mathtools}
\usepackage{braket}
\usepackage{esint}
\newcommand{\abs}[1]{{\left|#1\right|}}

\usepackage{relsize}

\usepackage{enumitem}
\usepackage{booktabs}
\usepackage{graphicx}
\usepackage{tikz}
\usetikzlibrary{patterns}
\usepackage{multicol}
\usepackage{caption}
\usepackage{bigints}
\usepackage[skins,theorems]{tcolorbox}
\tcbset{highlight math style={enhanced,
colframe=black,colback=white,arc=0pt,boxrule=1pt}}
\def\XXint#1#2#3{{\setbox0=\hbox{$#1{#2#3}{\int}$}
    \vcenter{\hbox{$#2#3$}}\kern-.5\wd0}}

\theoremstyle{definition}
\newtheorem{definizione}{Definition}[section]
\theoremstyle{plain}
\newtheorem{teorema}{Theorem}[section]

\newtheorem{lemma}[teorema]{Lemma}
\newtheorem{prop}[teorema]{Proposition}

\theoremstyle{definition}
\newtheorem{esempio}{Example}[section]
\newtheorem{oss}[esempio]{Remark}

\newtheorem*{open*}{Open problems}
\DeclareMathOperator{\R}{\mathbb{R}}

\makeatletter
\newcommand{\myfootnote}[2]{\begingroup
	\def\@makefnmark{}%
	\addtocounter{footnote}{-1}%
	\footnote{\textbf{#1} #2}
	\endgroup}
\makeatother

\definecolor{mygreen}{RGB}{0, 255, 0}

\usepackage{hyperref}
\hypersetup{linktoc=none, bookmarksnumbered, colorlinks=true, linkcolor=magenta, citecolor=cyan}

\title{Some shape functionals for the $k$-Hessian equation}
\author{Alba Lia Masiello, Francesco Salerno}
\date{}

\newcommand{\Addresses}{{
\bigskip 
  
   \medskip

    \textit{E-mail address}, A.L.~Masiello: \texttt{masiello@altamatematica.it} 
  
   \medskip 
 \textsc{Holder of a research grant from Istituto Nazionale di Alta Matematica "Francesco Severi" at Dipartimento di Matematica e Applicazioni "R. Caccioppoli", Via Cintia, Complesso Universitario Monte S.
Angelo, 80126 Napoli, Italy.}
   
   \medskip 
     \textit{E-mail address}, F.~Salerno (corresponding author): \texttt{f.salerno@ssmeridionale.it} 
   \medskip

 \textsc{Mathematical and Physical Sciences for Advanced Materials and Technologies, Scuola Superiore Meridionale, Largo San Marcellino 10, 80138 Napoli, Italy.}

 \par\nopagebreak 

}} 

\begin{document}
\maketitle
\begin{abstract}
  For a non-empty, bounded, open, and convex set of class $C^2$, we consider the Torsional Rigidity associated to the $k$-Hessian operator. We first prove P\'olya type lower bound for the $k$-Torsional Rigidity in any dimension; then, in order to investigate optimal sets in the P\'olya type inequality, we provide two quantitative estimates. 
\end{abstract}

\section{Introduction}
Let $\Omega\subset\mathbb{R}^n$, $n\geq 2$, be a non-empty, bounded, open, and convex set of class $C^2$. Let us consider the $k$-Torsional rigidity, denoted by $T_k(\Omega)$, and the Dirichlet eigenvalue $\lambda_k(\Omega)$ of the $k$-Hessian operator, whose variational characterizations are given by
\begin{equation}\label{varcar}
T_k(\Omega) = \max_{\substack{\varphi \in \mathcal{C}_k }} \frac{\left(\displaystyle{\int_\Omega -\varphi \, dx }\right)^{k+1}}{\displaystyle{\int_\Omega -\varphi S_k(D^2\varphi) \, dx }} \qquad \text{ and }\qquad\lambda_k(\Omega)= \min_{\substack{\varphi\in \mathcal{C}_k}} \frac{\displaystyle{\int_\Omega  -\varphi S_k(D^2\varphi) \, dx }}{\displaystyle{\int_\Omega  (-\varphi)^{k+1} \, dx }},
\end{equation}
where\footnote{The definition of $k$-Convex function can be found in Section \ref{hessian_operators}.}
\begin{equation*}
    \mathcal{C}_k=\{\varphi\in C^2(\Omega)\,:\, \varphi \text{ is $k$-convex}, \, \varphi=0 \text{ on } \partial\Omega\},
\end{equation*}
and $S_k(D^2 \varphi)$ is the \emph{$k$-Hessian operator}, defined as

\begin{equation}\label{kess}
    S_k(D^2 \varphi)= \sum_{1\le i_1< \dots < i_k\le n} \lambda_{i_1}\cdots\lambda_{i_k}, \quad k=1,\dots, n,
\end{equation}
being $\lambda_i$ the eigenvalues of the Hessian matrix of $u$. Let us notice that $S_k(D^2\varphi)$ is a second-order differential operator and it reduces to the Monge-Ampère operator for $k=n$ and to the Laplace operator for $k=1$.
The maximum and minimum in the previous two variational formulas are achieved, respectively, by the solutions to the following problems
\begin{equation}
    \label{torssk}
        \begin{cases}
            S_k(D^2u)= \binom{n}{k}& \text{ in $\Omega$, }\\
            u=0 & \text{ on $\partial \Omega$},
        \end{cases} 
    \end{equation}
\begin{equation}
   \label{autovalsk}
        \begin{cases}
            S_k(D^2u)=\lambda_k(\Omega)(-u)^{k}& \text{in } \Omega,\\
            u=0& \text{on } \partial\Omega.
        \end{cases}
    \end{equation}
The quantity $T_k(\Omega)$ is monotonically increasing with respect to set inclusion, while the quantity $\lambda_k(\Omega)$ is decreasing, and they satisfy the following scaling properties for all $t>0$
\begin{equation*}
     T_k(t\Omega)= t^{k{(n+2)}}T_k(\Omega),\qquad \lambda_k(t\Omega)= t^{-2k}\lambda_k(\Omega).
\end{equation*}
The case $k=1$, where the $k$-Hessian operator reduces to the Laplace operator, is well studied.  In \cite{polya1960} P\'olya gives some estimates on  the torsional rigidity $T(\Omega)=T_1(\Omega)$ of a bounded, open, planar convex set $\Omega$ in terms of geometrical quantities of the set itself.
In particular, he proves that,  in the aforementioned class of sets, the following inequality holds
\begin{equation}
    \label{pol}
    \frac{T(\Omega)P^2(\Omega)}{\abs{\Omega}^3}\ge \frac{1}{3}
\end{equation}
and he also proves that equality is asymptotically achieved by a sequence of thinning rectangles. P\'olya's proof relies on the so-called web-functions technique, which consists of considering test functions that depend on the \emph{distance to the boundary}. For more details about the web function approach, we refer to \cite{CFG_WEB}.
In the same paper and with the same techniques, the author also manages to prove that the first eigenvalue of the Laplacian with Dirichlet boundary conditions $\lambda(\Omega)=\lambda_1(\Omega)$ satisfies

\begin{equation}
    \label{polya:autoval}
    \frac{\lambda(\Omega)\abs{\Omega}^2}{P(\Omega)^2}\le \frac{\pi^2}{4},
\end{equation}
whenever $\Omega$ is a bounded, open, planar convex set and this inequality is asymptotically sharp on a sequence of thinning rectangles. 

As a step further,  Makai in \cite{makai} proves that the two shape functionals in \eqref{pol} and \eqref{polya:autoval} are bounded, providing the following upper bound for the Torsional rigidity
\begin{equation}\label{makai_intro}
   \frac{T(\Omega)P^2(\Omega)}{\abs{\Omega}^3}\le \frac{2}{3}, 
\end{equation}
which is sharp on a sequence of thinning triangles, and the following lower bound for the eigenvalue, which is sharp on the same sequence of sets,
\begin{equation}\label{polyaeigenvalue}
   \frac{\lambda( \Omega) \abs{\Omega}^2}{P^2(\Omega)}\ge  \frac{\pi^2}{16}.
\end{equation}

Over the years, this type of inequality connecting spectral quantities to geometrical ones has gained increasing interest. 
For example, estimates \eqref{pol} and \eqref{makai_intro} are generalized to the Torsional rigidity of the $p-$Laplacian in \cite{fragala_gazzola_lamboley2013}.
In \cite{gavitone_2014} the authors proved that the lower bound \eqref{pol} holds in every dimensions, 
and they also extend this result to the anisotropic case.

 We also recall that in \cite{buttazzo2020convex} the authors prove an upper bound 
\begin{equation*}
    \frac{T(\Omega)P{(\Omega)}^2}{\abs{\Omega}^3}\le \frac{2^{2n} n^{3n}}{\omega_n^2 }\frac{n}{n+2},
\end{equation*}
where $\omega_n$ is the Lebesgue measure of the unit ball. This upper bound is strictly larger than the one given in \eqref{makai_intro} for $n=2$, and that is why it is conjectured not to be optimal.

Moreover, strengthened versions of inequalities \eqref{pol}, \eqref{polya:autoval} can be found in \cite{AGS, grav}.

It is also worth mentioning the \emph{Hersh-Protter inequality} \cite{hersh,Protter}, which, in the same spirit of inequalities \eqref{pol} and \eqref{polya:autoval}, links the eigenvalue of the Laplacian of a convex set $\Omega$ with its \emph{inradius} $r(\Omega)$ (i.e. the radius of the biggest ball contained in $\Omega$), and states
$$\frac{\pi^2}{4}\le \lambda(\Omega)r(\Omega)^2\le \lambda(B).$$

\vspace{2mm}
In the present paper, we aim to establish upper and lower bounds for the functionals 
\begin{equation*}
    \begin{split}
        &(i)\quad \frac{T_k(\Omega)W_k(\Omega)P(\Omega)^k}{\abs{\Omega}^{2k+1}},\\&(ii)\quad \frac{\lambda_k(\Omega)\abs{\Omega}^{k+1}}{P(\Omega)^kW_k(\Omega)},
    \end{split}
\end{equation*}
that are the natural extension of the functional in \eqref{pol} and \eqref{polya:autoval} to the context of the $k$-Hessian operator,
 aiming to continue what was started in \cite{DellaPietra2012UpperBF}.

We start by considering the functional $(i)$, and we prove 

\begin{teorema} \label{teorema:torsione}
    Let $\Omega\subset\R^n$ be an open, bounded convex set of class $C^2$, then
    \begin{equation}\label{lower:torsion}
        \frac{T_k(\Omega)W_k(\Omega)P(\Omega)^k}{\abs{\Omega}^{2k+1}}\geq \frac{k^{k+1}}{n(2k+1)^k}.
    \end{equation}
\end{teorema}

In contrast to the Laplacian case, establishing whether inequality \eqref{lower:torsion} is sharp in the context of the $k$-Hessian operator is more challenging. Intuitively, one might expect that an optimal sequence is a sequence of thinning cylinders, and support for this intuition is provided by the following Theorems.

\begin{teorema}\label{lower:ags}
    Let $\Omega\subset\R^n$ be an open, bounded convex set of class $C^2$. Then
    \begin{equation*}
        \frac{T_k(\Omega)W_k(\Omega)P(\Omega)^k}{\abs{\Omega}^{2k+1}}-\frac{k^{k+1}}{n(2k+1)^k}\geq c_1(n,k)\alpha(\Omega),
    \end{equation*}
    where
    \begin{equation*}
        \alpha(\Omega)=\frac{w_\Omega}{D(\Omega)}.
    \end{equation*}
    Moreover, the constant is given by
    \begin{equation*}
        c_1(n,k)=\frac{k(2k+1)^{2+\frac{1}{k}}}{(3k+1)^{3+\frac{1}{k}}}\frac{k^{k+2}2(n-1)}{n^3(2k+1)^k}a(n)
    \end{equation*}
    and $a(n)$ is the constant in Proposition \ref{BonnesenFenchel}. 
\end{teorema}
This first Theorem ensures that an optimal set for \eqref{lower:torsion} cannot exist: it should have $0$ mean width, and this is absurd. On the other hand, it tells us that an optimal sequence must be a sequence of thinning domains. The shape of this thinning sequence is described by the following result.
\begin{teorema}\label{teorema:ags}
    Let $\Omega\subset\R^n$ be an open, bounded convex set of class $C^2$. Then
    \begin{equation*}
        \frac{T_k(\Omega)W_k(\Omega)P(\Omega)^k}{\abs{\Omega}^{2k+1}}-\frac{k^{k+1}}{n(2k+1)^k}\geq c_2(n,k){\beta({\Omega})}^\frac{2k+1}{k},
    \end{equation*}
    where
    \begin{equation*}
        \beta(\Omega)=\frac{P(\Omega)r(\Omega)}{\abs{\Omega}}-1,
    \end{equation*}
    and the constant is given by
    \begin{equation*}
        c_2(n,k)=\frac{k^{k+2}}{n(2k+1)^k(6n)^\frac{2k+1}{k}}.
    \end{equation*}
\end{teorema}

Theorem \ref{teorema:ags} gives us the following information: an optimal sequence must be optimal also for

$$\frac{P(\Omega)r(\Omega)}{\abs{\Omega}}-1\ge 0,$$
hence, for instance, it should be a sequence of thinning cylinders. Anyway, we are not able to estimate the torsional rigidity of thinning cylinders from above in order to confirm this intuition.

As far as the upper bound, we show in Example \ref{esempio} that shows that the functional $(i)$ is unbounded, at least for some values of $k$ and $n$.

For the functional $(ii)$, we prove the following upper bound.
\begin{teorema}\label{TeoAut}
 Let $\Omega\subset\R^n$ be an open, bounded, convex set of class $C^2$. Then,
\begin{equation*}
    \frac{\lambda_k(\Omega)\abs{\Omega}^{k+1}}{P(\Omega)^kW_k(\Omega)}\leq \binom{n}{k}\frac{n(2k+1)^k}{k^{k+1}}.
\end{equation*}

\end{teorema}
Finally, we show an example in which this functional goes to zero.

The paper is organized as follows: in Section \ref{Sec2} we recall some preliminary tools about convex geometry and the $k$-Hessian operator. In Section \ref{Sec3} we prove Theorems \ref{teorema:torsione}-\ref{lower:ags}-\ref{teorema:ags} on the torsion functional $(i)$. We conclude with Section \ref{Sec4} where Theorem \ref{TeoAut} is proved.

\section{Preliminaries}\label{Sec2}
For the content of this section, we will refer to \cite{Schneider2014}.
Throughout this article, $|\cdot|$ will denote the Euclidean norm in $\mathbb{R}^n$,
 while $\cdot$ is the standard Euclidean scalar product for  $n\geq2$. By $\mathcal{H}^d(\cdot)$, for $d\in [0,n)$, we denote the $d$-dimensional Hausdorff measure in $\mathbb{R}^n$.
 
 The perimeter of $\Omega$ in $\mathbb{R}^n$ will be denoted by $P(\Omega)$ and, if $P(\Omega)<\infty$, we say that $\Omega$ is a set of finite perimeter. In our case, $\Omega$ is a bounded, open, and convex set; this ensures us that $\Omega$ is a set of finite perimeter and that $P(\Omega)=\mathcal{H}^{n-1}(\partial\Omega)$. Moreover, if $\Omega$ is an open set with Lipschitz boundary, it holds

 \begin{teorema}[Coarea formula]
 Let $\Omega \subset \mathbb{R}^n$ be an open set. Let $f\in W^{1,1}_{\text{loc}}(\Omega)$ and let $u:\Omega\to\R$ be a measurable function. Then,
 \begin{equation}
   \label{coarea}
   {\displaystyle \int _{\Omega}u(x)|\nabla f(x)|dx=\int _{\mathbb {R} }dt\int_{\Omega\cap f^{-1}(t)}u(y)\, d\mathcal {H}^{n-1}(y)}.
 \end{equation}
 \end{teorema}
Some references for results relative to the sets of finite perimeter and for the coarea formula are, for instance, \cite{ambrosio2000functions, maggi2012sets}.

\vspace{2mm}

We provide the classical definitions and results that we need in the following. The reader can refer to \cite[Chapter 1]{Schneider2014} for more details.

Let $\Omega \subset \R^n$ be a non-empty, bounded, convex set, let $B$ be the unitary ball centered at the origin and $\rho > 0$. We can write the Steiner formula for the Minkowski sum $\Omega+ \rho B$ as
\begin{equation}
    \label{Steiner_formula}
    \abs{\Omega + \rho B} = \sum_{i=0}^n \binom{n}{i} W_i(\Omega) \rho^{i} .
\end{equation}
The coefficients $W_i(\Omega)$ are known in the literature as quermassintegrals of $\Omega$. In particular, $W_0(\Omega) = \abs{\Omega}$,  $nW_1(\Omega) = P(\Omega)$ and $W_n(\Omega) = \omega_n$ where $\omega_n$ is the measure of $B$.

Formula \eqref{Steiner_formula} can be generalized to every quermassintegral, obtaining
\begin{equation} \label{Steinerquermass}
    W_j(\Omega+\rho B) = \sum_{i=0}^{n-j} \binom{n-j}{i} W_{j+i}(\Omega) \rho^i, \qquad j=0, \ldots, n-1.
\end{equation}

If $\Omega$ is a convex set with $C^2$ boundary, the quermassintegrals can be written in terms of principal curvatures of $\Omega$. More precisely, denoting with $H_j$ the $j$-th normalized elementary symmetric function of the principal curvature $\kappa_1, \ldots, \kappa_{n-1}$ of $\partial \Omega$, i.e.
 \[
 H_0 = 1, \qquad \qquad H_j = \binom{n-1}{j}^{-1} \sum_{1 \leq i_1 < \ldots < i_j \leq n-1} \kappa_{i_1} \ldots \kappa_{i_j}, \qquad j = 1,\ldots,n-1,
 \]
 then, the quermassintegrals can be written as
 \begin{equation}
     \label{quermass_con_curvature}
     W_j(\Omega) = \frac{1}{n} \int_{\partial \Omega} H_{j-1} \, d \mathcal{H}^{n-1}, \qquad j = 1,\ldots, n-1.
 \end{equation}

Furthermore, Aleksandrov-Fenchel inequalities hold true
\begin{equation}
    \label{Aleksandrov_Fenchel_inequalities}
    \biggl( \frac{W_j(\Omega)}{\omega_n} \biggr)^{\frac{1}{n-j}} \geq \biggl( \frac{W_i(\Omega)}{\omega_n} \biggr)^{\frac{1}{n-i}}, \qquad 0 \leq i < j \leq n-1,
\end{equation}
where equality holds if and only if $\Omega$ is a ball. When $i=0$ and $j=1$, formula \eqref{Aleksandrov_Fenchel_inequalities} reduces to the classical isoperimetric inequality, i.e.
\[
    P(\Omega) \geq n \omega_n^{\frac{1}{n}} \abs{\Omega}^{\frac{n-1}{n}}.
\]

\begin{definizione}
    Let $\Omega$ be a non-empty, bounded, open, and convex set of $\R^n$. We define the distance function from the boundary, and we will denote it by $ d(\cdot, \partial \Omega):\Omega \to [0,+\infty[$, as follows 
 \begin{equation*}
     d(x,\partial\Omega):=\inf_{y\in\partial\Omega}\abs{x-y}.
 \end{equation*}

 The quantity
\begin{equation}
    \label{inradius:def}
r(\Omega)=\sup_{x\in\Omega}d(x,\partial\Omega),
\end{equation}
is the \emph{inradius} of the set $\Omega$, and it represents the radius of the largest ball that can be contained in $\Omega$.
\end{definizione}

\begin{definizione}
    \label{support}
Let $\Omega$ be a convex set. The \emph{support} function of $\Omega$ is defined as
\begin{equation*}
        h(\Omega, u):=\max_{x\in \Omega} (x\cdot u), \qquad u\in \mathbb{S}^{n-1}.
    \end{equation*}
\end{definizione}
The support function allows us to define, for a convex set $\Omega$, two geometrical quantities. The \emph{width function} $w(\Omega, \cdot)$ of $\Omega$
is defined as
$$w(\Omega, u):= h(\Omega, u) + h(\Omega, -u) \quad \text{for } u\in \mathbb{S}^{n-1}.$$

The quantity $w(\Omega, u)$ is the thickness of $\Omega$ in the direction $u$ and it represents the distance between
the two support hyperplanes of $\Omega$ orthogonal to $u$. The minimum of the width function,
$$w_\Omega:=\min_{u\in \mathbb{S}^{n-1}} w(\Omega, u)$$
is the minimal width of $\Omega$.

\begin{definizione}
    Let $\Omega_l$ be a sequence of non-empty, bounded, open, and convex sets of $\mathbb{R}^n$. We say that $\Omega_l$ is a sequence of thinning domains if
    \begin{equation}
        \dfrac{w_{\Omega_l}}{\text{diam}(\Omega_l)}\xrightarrow{l \to 0}0.
    \end{equation}
 
    In particular, if $l>0$ and  $C$ is a bounded, open, and convex set of $\R^{n-1}$ with unitary $(n-1)$-dimensional measure, then, if $l \to 0$, the sequence 
    \begin{equation}\label{thin_rect}
    \Omega_l = l^{-\frac{1}{n-1}}C \times \left[-\frac{l}{2}, \frac{l}{2}\right]
\end{equation}
    is called a sequence of thinning cylinders. Moreover, in the case $n=2$, the sequence \eqref{thin_rect} is called sequence of thinning rectangles.
    \begin{figure}[h]
        \centering
        \tikzset{every picture/.style={line width=0.75pt}} 

\begin{tikzpicture}[x=0.65pt,y=0.65pt,yscale=-1,xscale=1]

\draw  [fill={rgb, 255:red, 191; green, 186; blue, 186 }  ,fill opacity=0.42 ] (211.04,198.71) .. controls (196.02,190.03) and (190.87,183) .. (199.54,183) -- (438.23,183) .. controls (446.91,183) and (466.13,190.03) .. (481.15,198.71) -- (562.79,245.84) .. controls (577.82,254.52) and (582.97,261.55) .. (574.29,261.55) -- (335.6,261.55) .. controls (326.92,261.55) and (307.71,254.52) .. (292.68,245.84) -- cycle ;
\draw   (211.04,145.71) .. controls (196.02,137.03) and (190.87,130) .. (199.54,130) -- (438.23,130) .. controls (446.91,130) and (466.13,137.03) .. (481.15,145.71) -- (562.79,192.84) .. controls (577.82,201.52) and (582.97,208.55) .. (574.29,208.55) -- (335.6,208.55) .. controls (326.92,208.55) and (307.71,201.52) .. (292.68,192.84) -- cycle ;
\draw    (331,208.56) -- (330.78,260.69) ;
\draw    (578.33,205.9) -- (578.11,258.03) ;
\draw    (195.33,132.9) -- (195.11,185.03) ;
\draw  [dash pattern={on 4.5pt off 4.5pt}]  (443.33,131.23) -- (443.11,183.36) ;
\draw    (590.59,207.94) -- (590.59,261.36) ;
\draw [shift={(590.59,261.36)}, rotate = 270] [color={rgb, 255:red, 0; green, 0; blue, 0 }  ][line width=0.75]    (0,5.59) -- (0,-5.59)   ;
\draw [shift={(590.59,207.94)}, rotate = 270] [color={rgb, 255:red, 0; green, 0; blue, 0 }  ][line width=0.75]    (0,5.59) -- (0,-5.59)   ;

\draw (383.78,221.72) node [anchor=north west][inner sep=0.75pt]  [font=\tiny]  {$l^{-\frac{1}{n-1}} C$};
\draw (268.41,117.72) node [anchor=north west][inner sep=0.75pt]  [font=\tiny]  {$\Omega _{l}$};
\draw (594.96,232.09) node [anchor=north west][inner sep=0.75pt]  [font=\tiny]  {$l$};

\end{tikzpicture}
        \caption{Thinning cylinder.} \label{fig:M2}
    \end{figure}
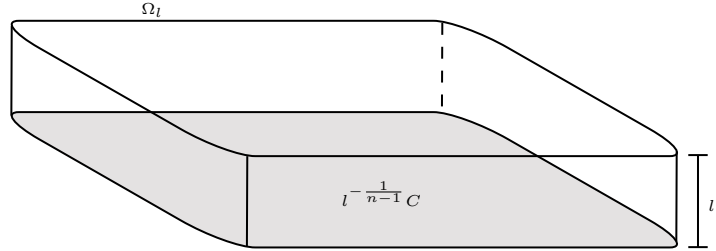
\end{definizione}
We remark that the distance function is concave, as a consequence of the convexity of $\Omega$.
The superlevel sets of the distance function
\begin{equation*}
    \Omega^t=\set{x\in \Omega \, : \, d(x,\partial\Omega)>t}, \qquad t\in[0, r(\Omega)]
\end{equation*}
 are called \emph{inner parallel sets}.
In what follows, we want to work with \emph{convex} functions, so we will use the following notation
\begin{equation*}
    \Omega_t=\set{x\in \Omega \, : \, -d(x,\partial\Omega)<t}, \qquad t\in[- r(\Omega), 0],
\end{equation*}

 \begin{equation*}
 \mu(t)= \abs{\Omega_t}, \qquad P(t)=P(\Omega_t),\qquad t\in[-r(\Omega),0].
 \end{equation*}
 By coarea formula, recalling that $\abs{\nabla d}=1$ almost everywhere, we have 
 $$\mu(t)=\int_{\Set{-d<t}} \,  dx = \int_{\Set{-d<t}} \frac{\abs{\nabla d}}{\abs{\nabla d}} \,  dx= \int_{-r(\Omega)}^{t} \frac{1}{\abs{\nabla d}} \int_{\Set{-d=s}} d\mathcal{H}^{n-1} \, ds= \int_{-r(\Omega)}^{t} P(s)\; ds;$$
hence, the function $\mu(t)$ is absolutely continuous, \emph{increasing} and its derivative is  
\begin{equation}\label{eq:dermu}
    \mu'(t)=P(t)\qquad a.e.
\end{equation}

{By the Brunn-Minkowski inequality (\cite[Theorem 7.4.5]{Schneider2014}) and the concavity of the distance function, the map
\begin{equation*}
    s \mapsto P(\Omega^s)^{\frac{1}{n-1}}
\end{equation*}
is concave in $[0,r({\Omega})]$, hence absolutely continuous in $(0,r({\Omega}))$. Moreover, there exists its right derivative at $0$ and it is negative, since $P(\Omega^s)^{\frac{1}{n-1}}$ is strictly monotone decreasing, hence almost everywhere differentiable. As a consequence, the function $P(s)$ is \emph{strictly increasing} and concave in $[-r(\Omega), 0]$.}

We recall the following results concerning the approximation of the distance to the boundary with $C^2$ functions (see \cite[Proposition 3.1]{DellaPietra2012UpperBF} for the proof)
\begin{prop}\label{approx}
Let $\Omega\subset\R^n$ be an open, bounded  convex set of class $C^2$. Then, there exists a
  sequence of functions $\{d_\varepsilon(x)\}_{\varepsilon>0}$, $x\in
  \bar{\Omega}$, such
  that:
  \begin{enumerate}
  \item $d_\varepsilon$ concave in $\Omega$, $d_\varepsilon\equiv 0$ on $\partial \Omega$ and
    $d_\varepsilon\in C^2(\Omega)\cap C(\bar\Omega)$;
  \item $0\le d_\varepsilon \le d$, and $d_\varepsilon\rightarrow d$ uniformly in
    $\bar \Omega$;
  \item $|\nabla d_\varepsilon|\le 1$ in $\Omega$.
  \end{enumerate}
\end{prop}

\subsection{\texorpdfstring{$k$}{k}-Hessian Operators} \label{hessian_operators}
Let $\Omega\subset \R^n$ be a bounded, connected, open set of class $C^2$, and let us consider $u \in C^2(\Omega)$. The \emph{$k$-Hessian operator} $S_k\left(D^2u\right)$ is defined
 as the $k$-th elementary symmetric function of $D^2u$, i. e.
 \begin{equation*}
    S_k(D^2 u)= \sum_{1\le i_1< \dots < i_k\le n} \lambda_{i_1}\cdots\lambda_{i_k}, \quad k=1,\dots, n,
\end{equation*}
being $\lambda_i$ the eigenvalue of the Hessian matrix of $u$.

In general, for $k\neq 1$, $S_k(D^2u)$ is not elliptic
unless we consider its restriction to the class of $k$-convex functions (see for instance \cite{caffarelli})
$$\mathcal{C}_k=\left\{u\in C^2(\Omega)\,:\,S_i(D^2u)\geq 0 \text{
in }\Omega, i=1,2,...,k \right\}.$$ In particular 
$\mathcal{C}=\mathcal{C}_n(\Omega)$ is equal to  the class of $C^2(\Omega)$
convex functions.
A function $u \in C^2(\Omega)$ is called strictly $k$-convex if the inequalities in the definition of the cone $\mathcal{C}_k(\Omega)$ hold strictly.

The
operator $S_k^{1/k}(D^2u)$, for $k=1,...,n$, is homogeneous of degree
$1$, so using the  notation
$$
S_k^{ij}(D^2u)
= \frac{\partial }{\partial u_{ij}}S_k(D^2u),$$
where the indices $i, j$ denote partial derivatives, Euler's identity for homogeneous functions yields
$$
S_k(D^2u) = \frac{1}{k} \sum_{i,j}S_k^{ij}(D^2u) u_{ij}.$$

One can show, by direct calculation, that  $\left(S_k^{1j}(D^2u), \dots, S_k^{nj}(D^2u)\right)$ is divergence-free, allowing $S_k(D^2u)$ to be expressed in divergence form

\begin{equation}
    \label{skdiv}
    S_k(D^2u)=\frac{1}{k}\left(S_k^{ij}(D^2u) u_j\right)_i.
\end{equation} 
For convenience, when the argument of $S_k^{ij}$ is omitted, we assume  $S_k^{ij}=S_k^{ij}(D^2u)$.

If $u\in C^2(\Omega)$ and $t$ is a regular value of $u$, the $k$-Hessian operator can be related to the $(k-1)$-th mean curvature
on the boundary of $\{u\le t\}$ via the identity (see, for instance,  \cite{reilly73, trudinger97})  
\begin{equation}
    \label{hksk}
   \binom{n-1}{k-1} H_{k-1}=\frac{S_k^{ij}u_iu_j}{\abs{\nabla u}^{k+1}}.
\end{equation}

\section{P\'olya torsion functional}\label{Sec3}
The aim of this section is to study the shape functional
\begin{equation}\label{FunzconTors}
    \frac{T_k(\Omega)W_k(\Omega)P(\Omega)^k}{\abs{\Omega}^{2k+1}}.
\end{equation}
We start considering the functional
\begin{equation}\label{aux:functional}
    F_k[w]:=\int_\Omega-w\,dx-\frac{1}{k+1}\int_\Omega-wS_k(D^2w)\,dx,
\end{equation}
where $w\in\mathcal{C}$, for which the following result holds.
\begin{lemma}\label{lemma1}
    For all $w\in\mathcal{C}$, it holds  $\displaystyle F_k[w]\leq\frac{k}{k+1}T_k(\Omega)^\frac{1}{k}$.
\end{lemma}
\begin{proof}
    Let $w\in\mathcal{C}_k$ and $\lambda\geq0$, let us consider the function 
    \begin{equation}\label{flambda}
        f(\lambda):=F_k[\lambda w]=\lambda\int_\Omega -w\,dx-\frac{\lambda^{k+1}}{k+1}\int_\Omega-wS_k(D^2w)\,dx\leq \max_{\lambda\geq 0}f(\lambda).
    \end{equation}
    If we optimize with respect to $\lambda>0$,
    \begin{equation*}
        f'(\lambda)=\int_\Omega-w\,dx-\lambda^k\int_\Omega-wS_k(D^2w)\,dx
    \end{equation*}
    we find
    \begin{equation*}
        \lambda_{\text{max}}=\frac{\left(\int_\Omega-w\,dx\right)^\frac{1}{k}}{\left(\int_\Omega-wS_k(D^2w)\,dx\right)^\frac{1}{k}}.
    \end{equation*}
    Hence, we can conclude
    \begin{equation*}
        F_k[w]\leq F_k[\lambda_\text{max}w]=\frac{k}{k+1}T_k(\Omega)^\frac{1}{k}.
    \end{equation*}
\end{proof}
Let us consider the class
\begin{equation*}
    \mathcal{C}^-=\{g:[M,0]\rightarrow\R \text{ : $g(0)=0$, $g$ increasing and convex}\}.
\end{equation*}
\begin{oss}
    If $g\in\mathcal{C}^-$ and $u$ is convex, then $g\circ u$ is also convex.
\end{oss}
\begin{lemma}\label{lemma2}
    Let $u\in\mathcal{C}_n$ and let $g\in\mathcal{C^-}$, then it holds
    \begin{equation*}
        \left\{\frac{k+1}{k}\max_{g\in\mathcal{C}^-}F_k[g\circ u]\right\}^k=k\left[\bigintsss_M^0\frac{\mu(t)^{1+\frac{1}{k}}}{\left(\int_{\{u=t\}}\frac{S_k^{ij}(D^2u)u_i u_j}{|\nabla u|} \right)^\frac{1}{k}} \right]^k.
    \end{equation*}
\end{lemma}
\begin{proof}
Let $u$ be a convex function that vanishes on $\partial\Omega$ and let $M$ be its minimum. We want to rewrite $F_k[g\circ u]$ in terms of $g$, so the first integral becomes
\begin{equation*} 
\int_\Omega-g\circ u\,dx=\int_\Omega \left(\int_u^0 g'(t)\,dt\right) \, dx=\int_M^0 g'(t)\mu(t)\,dt.
    \end{equation*}
    On the other hand, the second integral can be rewritten thanks to \eqref{skdiv} and the coarea formula
    \begin{equation*}
        \begin{split}
            \int_\Omega(-g\circ u)S_k(D^2(g\circ u))\,dx&=\frac{1}{k}\int_\Omega (-g\circ u)(S_k^{ij}(D^2(g\circ u))(g\circ u)_i)_j\,dx\\
            &=\frac{1}{k}\int_\Omega S_k^{ij}(D^2(g\circ u))(g\circ u)_i(g\circ u)_j\,dx\\
            &=\frac{1}{k}\int_{g(M)}^0\int_{\{g(u)=s\}}\frac{S_k^{ij}(D^2(g\circ u))(g\circ u)_i(g\circ u)_j}{|\nabla g\circ u|}\,d\mathcal{H}^{n-1}\,ds.  
             \end{split}
    \end{equation*}
    Thanks to formula \eqref{hksk}, we get
    \begin{equation*}
        \begin{split}
        \int_\Omega(-g\circ u)S_k(D^2(g\circ u))\,dx&=\frac{1}{k}\int_{g(M)}^0\int_{\{g(u)=s\}}H_{k-1}(\{g\circ u=s\})|\nabla g\circ u|^k\,d\mathcal{H}^{n-1}\,ds\\
            &=\frac{1}{k}\int_M^0g'(t)^{k+1}\int_{\{u=t\}}H_{k-1}(\{u=t\})|\nabla u|^k\,d\mathcal{H}^{n-1}\,dt\\
            &=\frac{1}{k}\int_M^0g'(t)^{k+1}\int_{\{u=t\}}\frac{S_k^{ij}(D^2u)u_iu_j}{|\nabla u|}\,d\mathcal{H}^{n-1}\,dt.
       \end{split}
       \end{equation*}
    To conclude with the claim, we can proceed as in \eqref{flambda}, and we obtain
    \begin{equation*}
        \begin{split}
            \max_{g\in \mathcal{C}^-}F_k[g\circ u]&= \max_{g\in \mathcal{C}^-}\int_M^0\left[g'(t)\mu(t)-\frac{1}{k(k+1)}g'(t)^{k+1}\int_{\{u=t\}}\frac{S_k^{ij}(D^2u)u_iu_j}{\abs{\nabla u}}\,d\mathcal{H}^{n-1} \right]dt\\
            &=\max_{s\geq 0}\left\{s\mu(t)-\frac{1}{k(k+1)}s^{k+1}\int_{\{u=t\}}\frac{S_k^{ij}(D^2u)u_iu_j}{\abs{\nabla u}}\,d\mathcal{H}^{n-1} \right\}
        \end{split}
    \end{equation*}
    where the optimal $s$ is
    \begin{equation*}
        s_{max}=\left(\frac{k\mu(t)}{\int_{\{u=t\}}\frac{S_k^{ij}(D^2u)u_iu_j}{\abs{\nabla u}}\,d\mathcal{H}^{n-1}} \right)^\frac{1}{k}.
    \end{equation*}
\end{proof}
Before proving Theorem \ref{teorema:torsione}, we need the following lemma.
\begin{lemma}\label{concInfLap}
    Let $\Omega\subseteq\R^n$ be a convex set, and let $f:\Omega\rightarrow\R$ be a concave function, then $(D^2f\cdot\nabla f,\nabla f)\leq0$.
\end{lemma}
\begin{proof}
    Let $\overline{x}\in\Omega$ and let us consider the curve $\gamma(t)=\overline{x}+t\nabla f(\overline{x})$, $t\in[0,1]$ and $g(t):=f(\gamma(t))$. Then
    \begin{equation*}
        \begin{split}
            g'(t)&=\gamma'(t)\nabla f(\gamma(t)),\\
            g''(t)&=2(D^2f(\gamma(t))\cdot\nabla f(\gamma(t)),\nabla f(\gamma(t))).
        \end{split}
    \end{equation*}
   Since $g$ is the restriction of $f$ to a line, it is concave and then
    \begin{equation*}
        0\geq g''(0)=2(D^2f(\overline{x})\cdot\nabla f(\overline{x}),\nabla f(\overline{x})).
    \end{equation*}
\end{proof}
We are now ready to prove the lower bound for the Torsion functional \eqref{FunzconTors}.

\begin{proof}[Proof of Theorem \ref{teorema:torsione}]
Let $d_\varepsilon$ be the function defined in Proposition \ref{approx}, we can consider $F_k[g(-d_\varepsilon)]$. Let us denote by $M_\varepsilon$ the maximum of $d_\varepsilon$. 
From Lemmas \ref{lemma1}-\ref{lemma2} we have 
\begin{equation}\label{BoundTor}
    T_k(\Omega)\geq \sup_{g\in \mathcal{C}^-}\frac{\left(\int_\Omega -g(d_\varepsilon)\,dx \right)^{k+1}}{\int_\Omega -g(d_\varepsilon)S_k(D^2g(d_\varepsilon))\,dx}=k\left[\int_{-M_\varepsilon}^0\frac{\mu_\varepsilon(t)^{1+\frac{1}{k}}}{\left(\int_{\{d_\varepsilon=t\}}\frac{S_k^{ij}(D^2d_\varepsilon)(d_\varepsilon)_i(d_\varepsilon)_j}{\abs{\nabla d_\varepsilon}}\,d\mathcal{H}^{n-1} \right)^\frac{1}{k}}\,dt\right]^k.
\end{equation}
We remark that, by \eqref{hksk}, point $3$ in Proposition \eqref{approx} and \eqref{quermass_con_curvature}, it follows
\begin{equation*}
    \begin{split}
        \int_{\{d_\varepsilon=t\}}\frac{S_k^{ij}(D^2d_\varepsilon)(d_\varepsilon)_i(d_\varepsilon)_j}{\abs{\nabla d_\varepsilon}}\,d\mathcal{H}^{n-1}&=\int_{\{d_\varepsilon=t\}}H_{k-1}(\{d_\varepsilon=t\})\abs{\nabla d_\varepsilon}^k \,d\mathcal{H}^{n-1}\\
        &\leq\int_{\{d_\varepsilon=t\}}H_{k-1}(\{d_\varepsilon=t\})\,d\mathcal{H}^{n-1}=nW_k(\{d_\varepsilon=t\})\\&\leq nW_k(\Omega),
    \end{split}
\end{equation*}
where the last inequality is a consequence of the monotonicity of the quermassintegral with respect to inclusion; this allows us to rewrite \eqref{BoundTor}
\begin{equation}\label{step1}
    \left(\frac{T_k(\Omega)}{k}\right)^\frac{1}{k}\geq \frac{1}{(nW_k(\Omega))^\frac{1}{k}}\int_{-M_\varepsilon}^0\mu_\varepsilon(t)^{1+\frac{1}{k}}\, dt.
\end{equation}

We now integrate by parts the right-hand side, obtaining

\begin{equation}\label{TorBound2}
    \begin{split}
        \left(\frac{T_k(\Omega)}{k}\right)^\frac{1}{k}&\geq \frac{1}{(nW_k(\Omega))^\frac{1}{k}}\int_{-M_\varepsilon}^0\mu_\varepsilon(t)^{1+\frac{1}{k}}\frac{\mu_\varepsilon'(t)}{\mu_\varepsilon'(t)}\,dt\\&=\frac{k}{2k+1}\frac{1}{(nW_k(\Omega))^\frac{1}{k}}\int_{-M_\varepsilon}^0\frac{d}{dt}\left[(\mu_\varepsilon(t))^{2+\frac{1}{k}}\right]\frac{dt}{\mu_\varepsilon'(t)}\\
        &=\frac{k}{2k+1}\frac{1}{(nW_k(\Omega))^\frac{1}{k}}\frac{\mu_\varepsilon(0)^{2+\frac{1}{k}}}{\mu_\varepsilon'(0)}+\frac{1}{(nW_k(\Omega))^\frac{1}{k}}\int_{-M_\varepsilon}^0\frac{\mu_\varepsilon^{2+\frac{1}{k}}(t)\mu_\varepsilon''(t)}{(\mu_\varepsilon'(t))^2}\,dt.
    \end{split}
\end{equation}
The last term in the previous equality is nonnegative, and this follows from the fact that  $\mu_\varepsilon''$ is nonnegative. 

Indeed, applying the Coarea Formula, we have that
\begin{equation*}
    \begin{split}
        \mu_\varepsilon'(t)&=\int_{\{d_\varepsilon=t\}}\frac{1}{\abs{\nabla d_\varepsilon}}\,d\mathcal{H}^{n-1}=\int_{\{d_\varepsilon=t\}}\frac{\nu\cdot\nu}{\abs{\nabla d_\varepsilon}}\,d\mathcal{H}^{n-1}=\int_{\{d_\varepsilon<t\}}\mathrm{div}\left(\frac{\nu}{\abs{\nabla d_\varepsilon}} \right) \,d\mathcal{H}^{n-1}\\
        &=\int_{-M_\varepsilon}^0\int_{\{d_\varepsilon=t\}}\frac{1}{\abs{\nabla d_\varepsilon}}\mathrm{div}\left(\frac{\nu}{\abs{\nabla d_\varepsilon}} \right)\,d\mathcal{H}^{n-1},
    \end{split}
\end{equation*}
and then
{
\begin{equation*}
    \mu_\varepsilon''(t)=\int_{\{d_\varepsilon=t\}}\frac{1}{\abs{\nabla d_\varepsilon}}\mathrm{div}\left(\frac{\nu}{\abs{\nabla d_\varepsilon}} \right)\,d\mathcal{H}^{n-1}=\underbrace{\int_{\{d_\varepsilon=t\}}\frac{H_1}{\abs{\nabla d_\varepsilon}^2}\,d\mathcal{H}^{n-1}}_{\geq 0}+\int_{\{d_\varepsilon=t\}}\nabla\left(\frac{1}{\abs{\nabla d_\varepsilon}}\right)\cdot\frac{\nu}{\abs{\nabla d_\varepsilon}}\,d\mathcal{H}^{n-1}.
\end{equation*}}
The first integral is positive because of the convexity of the set $\{-d_\varepsilon<t\}$, while the last term is
\begin{equation*}
    \int_{\{d_\varepsilon=t\}}\nabla\left(\frac{1}{\abs{\nabla d_\varepsilon}}\right)\cdot\frac{\nu}{\abs{\nabla d_\varepsilon}}\,d\mathcal{H}^{n-1}=\int_{\{d_\varepsilon=t\}}\frac{-(\nabla d_\varepsilon\times D^2d_\varepsilon)\cdot\nabla d_\varepsilon}{\abs{\nabla d_\varepsilon}^5} \,d\mathcal{H}^{n-1}
\end{equation*}
and it nonnegative from Lemma \ref{concInfLap}.

\begin{figure}
    \centering
    \tikzset{every picture/.style={line width=0.75pt}} 

\begin{tikzpicture}[x=1pt,y=1pt,yscale=-1,xscale=1]

\draw [color={rgb, 255:red, 74; green, 144; blue, 226 }  ,draw opacity=1 ]   (338.02,110) .. controls (274.72,139.83) and (275.52,140.02) .. (211.43,110) ;
\draw  [color={rgb, 255:red, 65; green, 117; blue, 5 }  ,draw opacity=1 ] (274.73,132.44) -- (211.43,109.6) -- (338.02,109.6) -- cycle ;
\draw    (185.58,109.9) -- (374,110) ;
\draw [shift={(376,110)}, rotate = 180.03] [color={rgb, 255:red, 0; green, 0; blue, 0 }  ][line width=0.75]    (10.93,-3.29) .. controls (6.95,-1.4) and (3.31,-0.3) .. (0,0) .. controls (3.31,0.3) and (6.95,1.4) .. (10.93,3.29)   ;
\draw    (274.73,90) -- (274.73,188) ;
\draw [shift={(274.73,190)}, rotate = 270] [color={rgb, 255:red, 0; green, 0; blue, 0 }  ][line width=0.75]    (10.93,-3.29) .. controls (6.95,-1.4) and (3.31,-0.3) .. (0,0) .. controls (3.31,0.3) and (6.95,1.4) .. (10.93,3.29)   ;
\draw  [color={rgb, 255:red, 128; green, 128; blue, 128 }  ,draw opacity=1 ] (274.73,152.48) -- (211.43,110) -- (338.02,110) -- cycle ;
\draw [color={rgb, 255:red, 245; green, 166; blue, 35 }  ,draw opacity=1 ]   (210.81,109.6) -- (338.67,109.6) ;

\draw (316,127.4) node [anchor=north west][inner sep=0.75pt]  [font=\tiny,color={rgb, 255:red, 128; green, 128; blue, 128 }  ,opacity=1 ]  {-$d$};
\draw (223,125.4) node [anchor=north west][inner sep=0.75pt]  [font=\tiny,color={rgb, 255:red, 74; green, 144; blue, 226 }  ,opacity=1 ]  {-$d_\varepsilon $};
\draw (296,113.4) node [anchor=north west][inner sep=0.75pt]  [font=\tiny,color={rgb, 255:red, 65; green, 117; blue, 5 }  ,opacity=1 ]  {-$c_{\varepsilon }$};
\draw (263,182.4) node [anchor=north west][inner sep=0.75pt]  [font=\tiny]  {$y$};
\draw (367,113.4) node [anchor=north west][inner sep=0.75pt]  [font=\tiny]  {$x$};
\draw (257,102.4) node [anchor=north west][inner sep=0.75pt]  [font=\tiny,color={rgb, 255:red, 245; green, 166; blue, 35 }  ,opacity=1 ]  {$\Omega $};

\end{tikzpicture}
    \caption{}
    \label{Figura 1}
\end{figure}

To conclude the proof, we claim that, as $\varepsilon$ goes to $0$
\begin{equation*}
    \mu_\varepsilon(0)\xrightarrow{\varepsilon\rightarrow0}\abs{\Omega}, \quad \mu_\varepsilon'(0)\xrightarrow{\varepsilon\rightarrow 0}P(\Omega).
\end{equation*}
The convergence of $\mu_\varepsilon(0)$ follows from the uniform convergence of $d_\varepsilon$ to $d(\cdot, \partial\Omega)$. To prove that $\mu_\varepsilon'(0)\xrightarrow{\varepsilon\rightarrow 0}P(\Omega)$, we observe that the function $d_\varepsilon$ is concave in $\Omega$, hence there exists a function $C_\varepsilon$  (see Figure \ref{Figura 1}) such that

\begin{itemize}
    \item $C_\varepsilon\le d_\varepsilon$;
    \item The graph of $C_\varepsilon$ is a cone of basis $\Omega$ and height $M_\varepsilon$;
    \item $\mu_{C_\varepsilon}(t)=\abs{\set{-C_\varepsilon<t}}\le \mu_\varepsilon(t)\le \mu(t)=\abs{\set{-d(\cdot,\partial\Omega)<t}}$;
    \item $\set{-C_\varepsilon<t}=\left(1+\frac{t}{M_\varepsilon}\right)\Omega$.
\end{itemize}
 Now, we can compute the derivative
\begin{equation*}
    \abs{\Omega}\frac{(1+ h/M_\varepsilon)^n-1}{h}=\frac{\mu_{C_\varepsilon}(h)-\mu_{C_\varepsilon}(0)}{h}\leq \frac{\mu_{\varepsilon}(h)-\mu_{\varepsilon}(0)}{h}\leq \frac{\mu(h)-\mu(0)}{h},
\end{equation*}
and we obtain
$$\frac{P(\Omega) r(\Omega)}{M_\varepsilon}\le\frac{n\abs{\Omega}}{M_\varepsilon}\le \mu'_\varepsilon(0)\le P(\Omega).$$
Since $M_\varepsilon\rightarrow r(\Omega)$ as $\varepsilon\rightarrow0^+$, we have that $\mu'_\varepsilon(0)\rightarrow P(\Omega)$.
Then from \eqref{TorBound2} we find
\begin{equation*}
    T_k(\Omega)\geq \frac{k^{k+1}}{n(2k+1)^k}\frac{\abs{\Omega}^{2k+1}}{W_k(\Omega)P(\Omega)^k}\implies \frac{T_k(\Omega)W_k(\Omega)P(\Omega)^k}{\abs{\Omega}^{2k+1}}\geq \frac{k^{k+1}}{n(2k+1)^k}.
\end{equation*}
\end{proof}

Let us observe that in the case $k=1$ we recover exactly the result by P\'olya \eqref{pol}. One might wonder if an upper bound as in \eqref{makai_intro} can hold.  In the next example, we show that, in the case of $n=k=2$, the P\'olya functional for the Hessian equation is not bounded from above. 

Indeed, the derivation of an upper bound might be strongly influenced by the choice of 
$k$ in relation to the dimension, and this dependence is not unexpected, as it is already manifest in \cite{Gavitone2007}.

\begin{esempio}\label{esempio}
    Let us consider $k=n=2$. In this case the functional \eqref{FunzconTors} is
    \begin{equation*}
        \frac{T_2(\Omega)P(\Omega)^2W_2(\Omega)}{\abs{\Omega}^5},
    \end{equation*}
    and we consider $ \mathcal{E}_\varepsilon$ the sequence of ellipses
    \begin{equation*}
        \mathcal{E}_\varepsilon:\, \left\{\ (x,y)\in \R^2 \frac{x^2}{\varepsilon^2}+\varepsilon^2y^2\leq 1\right\}.
    \end{equation*}
    We recall that the $n$-Torsion (the one related to the Monge-Ampère operator) of an ellipse is given in \cite{newisop_BNT}
    \begin{equation*}
        T(\mathcal{E}_\varepsilon)=\frac{\abs{ \mathcal{E}_\varepsilon}^{n+2}}{(n+2)^n\omega_n^2}.
    \end{equation*}
    Then, in this case, we have
    \begin{equation*}
        W_2(\mathcal{E}_\varepsilon)=\omega_2=\pi,\, \,\abs{\mathcal{E}}=\pi,\, \, T(\mathcal{E}_\varepsilon)=\frac{\pi^2}{16},\, \, P(\mathcal{E}_\varepsilon)\approx 2\pi\sqrt{\frac{\varepsilon^4+1}{2\varepsilon^2}}.
    \end{equation*}
    Then, as $\varepsilon\rightarrow0^+$, the shape functional diverges.
\end{esempio}

Once we have proven the lower bound \eqref{lower:torsion}, one can ask if it is sharp, at least asymptotically, as in the case $k=1$. Of course, in our case it can be more difficult to compute the $k$-torsional rigidity of a set.

In \cite{AGS} the authors add a reminder term in inequality \eqref{pol} which takes into account the thickness of the domain. In particular, this term vanishes as the domain loses one dimension, as in the case of the thinning cylinders that, asymptotically, attain the equality in \eqref{pol}. Following their strategy, we are able to do the same with the estimate proved in Theorem \ref{teorema:torsione}, and this will be the content of Theorem \ref{lower:ags}. Before doing this, we need to recall some preliminary. 
In the next proposition, we recall two classical inequalities in convex analysis, which can be found in Chapter 10, Section 44 of \cite{BonnesenFenchel}, formulae $(6)$ and $(9)$. 
\begin{prop}\label{BonnesenFenchel}
    Let $\Omega$ be a bounded, open convex set of $\R^n$. Then, the following inequalities hold true
    \begin{equation*}
        \begin{split}
            (i)&\quad P(\Omega)\leq n\omega_n\left(\frac{D(\Omega)}{2}\right)^{n-1},\\
            (ii)&\quad r(\Omega)\geq a(n)w_\Omega,
        \end{split}
    \end{equation*}
    where the constant is
    \begin{equation*}
        a(n)=\begin{cases}
            \frac{\sqrt{n+2}}{2n+2} & \text{ if n is even}\\
            \frac{1}{2\sqrt{n}} & \text{ if n is odd.}
        \end{cases}
    \end{equation*}
\end{prop}
\begin{oss}
        In the previous proposition some rigidity results hold true. More in detail, if $n=2$, equality occurs in $(i)$ if and only if $\Omega$ is a set with constant width function, and equality occurs in $(ii)$ if and only if $\Omega$ is an equilateral triangle. For $n\geq3$, the equality in $(i)$ implies that $\Omega$ is a ball, while the equality in $(ii)$ is achieved, for istance, by a regular simplex.
\end{oss}
We report without proof the following lemma that can be found in \cite[Lemma 2.4]{AGS}.
\begin{lemma}\label{AmatoTonto}
    Let $\Omega$ be a bounded, open, and convex set of $\R^n$. Then
    \begin{equation*}
        P(t)\leq P(\Omega)-c_n\frac{\abs{\Omega}-\mu(t)}{P(\Omega)^\frac{1}{n-1}},
    \end{equation*}
    for all $t\in[0,r(\Omega)]$, where the constant is given by
    \begin{equation}\label{c_n}
        c_n=\frac{n-1}{n^\frac{n-2}{n-1}}\omega_n^\frac{1}{n-1}.
    \end{equation}
\end{lemma}

We are now in position to prove Theorem \ref{lower:ags}.

\begin{proof}[Proof of Theorem \ref{lower:ags}]
    Let us start by sending $\varepsilon$ to $0$ in \eqref{step1}, 
    \begin{equation}\label{inizio1}
    \begin{split}
     \left(\frac{T_k(\Omega)}{k}\right)^\frac{1}{k}&\geq \frac{1}{(nW_k(\Omega))^\frac{1}{k}}\int_{-r(\Omega)}^0\mu(t)^{1+\frac{1}{k}}\, dt\\
        &=\frac{1}{(nW_k(\Omega))^\frac{1}{k}}\int_{-r(\Omega)}^{\overline{t}}\mu(t)^{1+\frac{1}{k}}\, dt+\frac{1}{(nW_k(\Omega))^\frac{1}{k}}\int_{\overline{t}}^0\mu(t)^{1+\frac{1}{k}}\, dt,
    \end{split}
    \end{equation}
    where $\overline{t}\in(-r(\Omega),0)$ satisfies $\mu(\overline{t})=\Tilde{c}\abs{\Omega}$, with $\tilde c\in(0,1)$ to be choosen later. Then, from Lemma \ref{AmatoTonto}
    \begin{equation*}
        \begin{split}
            \left(\frac{T_k(\Omega)}{k}\right)^\frac{1}{k}&\geq \frac{1}{P(\Omega)(nW_k(\Omega))^\frac{1}{k}}\int_{\overline{t}}^{0}\mu(t)^{1+\frac{1}{k}}\mu'(t)\,dt+\frac{1}{(nW_k(\Omega))^\frac{1}{k}}\int_{-r(\Omega)}^{\overline{t}}\mu(t)^{1+\frac{1}{k}}\frac{\mu'(t)}{P(t)}\,dt\\
            &\geq \frac{k}{2k+1}\frac{1}{P(\Omega)(nW_k(\Omega))^\frac{1}{k}}\int_{\overline{t}}^{0}\frac{d}{dt}\left(\mu(t)^{2+\frac{1}{k}}\right)\,dt\\
            &+\frac{k}{(2k+1)(nW_k(\Omega))^\frac{1}{k}\left(P(\Omega)-c_n\frac{\abs{\Omega}-\mu(\overline{t})}{P(\Omega)^\frac{1}{n-1}} \right)}\int_{-r(\Omega)}^{\overline{t}}\frac{d}{dt}\left(\mu(t)^{2+\frac{1}{k}}\right)\,dt\\
            &=\frac{k}{2k+1}\left\{\frac{\abs{\Omega}^{2+\frac{1}{k}}-\mu(\overline{t})^{2+\frac{1}{k}}}{P(\Omega)(nW_k(\Omega))^\frac{1}{k}}+\frac{\mu(\overline{t})^{2+\frac{1}{k}}}{(nW_k(\Omega))^\frac{1}{k}\left(P(\Omega)-c_n\frac{\abs{\Omega}-\mu(\overline{t})}{P(\Omega)^\frac{1}{n-1}} \right)} \right\}.
        \end{split}
    \end{equation*}
    For the second term, since 
    \begin{equation*}
        \frac{1}{1-a}\geq 1+a\quad\forall\,a\in[0,1),
    \end{equation*}
    we have that
    \begin{equation*}
        \begin{split}
            \frac{\mu(\overline{t})^{2+\frac{1}{k}}}{(nW_k(\Omega))^\frac{1}{k}\left(P(\Omega)-c_n\frac{\abs{\Omega}-\mu(\overline{t})}{P(\Omega)^\frac{1}{n-1}} \right)}&=\frac{\mu(\overline{t})^{2+\frac{1}{k}}}{(nW_k(\Omega))^\frac{1}{k}P(\Omega)\left(1-c_n\frac{\abs{\Omega}-\mu(\overline{t})}{P(\Omega)^\frac{n}{n-1}} \right)}\\
            &\geq \frac{\mu(\overline{t})^{2+\frac{1}{k}}}{P(\Omega)(nW_k(\Omega))^\frac{1}{k}}\left(1+c_n\frac{\abs{\Omega}-\mu(\overline{t})}{P(\Omega)^\frac{n}{n-1}} \right).
        \end{split}
    \end{equation*}
    Then, recalling that $\mu(\overline{t})=\tilde c\abs{\Omega}$,
    \begin{equation*}
        \begin{split}
            \left(\frac{T_k(\Omega)}{k}\right)^\frac{1}{k}&\geq \frac{k}{2k+1}\left\{\frac{\abs{\Omega}^{2+\frac{1}{k}}-\mu(\overline{t})^{2+\frac{1}{k}}}{P(\Omega)(nW_k(\Omega))^\frac{1}{k}}+\frac{\mu(\overline{t})^{2+\frac{1}{k}}}{P(\Omega)(nW_k(\Omega))^\frac{1}{k}}\left(1+c_n\frac{\abs{\Omega}-\mu(\overline{t})}{P(\Omega)^\frac{n}{n-1}} \right) \right\}\\
            &=\frac{k}{2k+1}\left\{\frac{\abs{\Omega}^{2+\frac{1}{k}}}{P(\Omega)(nW_k(\Omega))^\frac{1}{k}}+c_n\frac{\tilde c^{2+\frac{1}{k}}(1-\tilde c)\abs{\Omega}^{3+\frac{1}{k}}}{P(\Omega)^{1+\frac{n}{n-1}}(nW_k(\Omega))^\frac{1}{k}} \right\}.
        \end{split}
    \end{equation*}
    Now we want to maximize $(1-\tilde c)\tilde c^{2+\frac{1}{k}}$, that gives us $\tilde c=\frac{2k+1}{3k+1}$. Hence we have
    \begin{equation*}
        \left(\frac{T_k(\Omega)}{k}\right)^\frac{1}{k}\geq \frac{k}{2k+1}\left\{\frac{\abs{\Omega}^{2+\frac{1}{k}}}{P(\Omega)(nW_k(\Omega))^\frac{1}{k}}+c_n\frac{k(2k+1)^{2+\frac{1}{k}}}{(3k+1)^{3+\frac{1}{k}}}\frac{\abs{\Omega}^{3+\frac{1}{k}}}{P(\Omega)^{1+\frac{n}{n-1}}(nW_k(\Omega))^\frac{1}{k}} \right\}.
    \end{equation*}
    Now, recalling that $(a+b)^k\geq a^k+ka^{k-1}b$, where $a,\,b\geq0$, from the previous inequality follows that
    \begin{equation*}
        \frac{T_k(\Omega)}{k}\geq \left(\frac{k}{2k+1}\right)^k\left\{\frac{\abs{\Omega}^{2k+1}}{P(\Omega)^knW_k(\Omega)}+\frac{c_nk^2(2k+1)^{2+\frac{1}{k}}}{(3k+1)^{3+\frac{1}{k}}}\frac{\abs{\Omega}^{2k+2}}{P(\Omega)^{k+\frac{n}{n-1}}nW_k(\Omega)} \right\}.
    \end{equation*}
    Applying the following inequality
    \begin{equation*}
        \frac{\abs{\Omega}}{P(\Omega)}\geq \frac{r(\Omega)}{n},
    \end{equation*}
    from Proposition \ref{BonnesenFenchel}, we have
    \begin{equation*}
        \begin{split}
            \frac{T_k(\Omega)W_k(\Omega)P(\Omega)^k}{\abs{\Omega}^{2k+1}}-\frac{k^{k+1}}{n(2k+1)^k}&\geq \frac{k^{k+3}c_n(2k+1)^{2+\frac{1}{k}}}{n(2k+1)^k(3k+1)^{3+\frac{1}{k}}}\frac{\abs{\Omega}}{P(\Omega)^\frac{n}{n-1}}\\&\geq\frac{k^{k+3}c_n(2k+1)^{2+\frac{1}{k}}}{n^2(2k+1)^k(3k+1)^{3+\frac{1}{k}}}\frac{r(\Omega)}{P(\Omega)^\frac{1}{n-1}}\\
            &\geq \frac{k^{k+3}c_n(2k+1)^{2+\frac{1}{k}}}{n^2(2k+1)^k(3k+1)^{3+\frac{1}{k}}}\frac{2}{(n\omega_n)^\frac{1}{n-1}}a(n)\alpha(\Omega).
        \end{split}
    \end{equation*}
    Then, from the definition of $c_n$ \eqref{c_n} we have the thesis.
\end{proof}

Theorem \ref{lower:ags} ensures that a minimizing sequence for inequality \eqref{lower:torsion} must be a sequence of thinning domains. To obtain more information about the shape of this sequence we need the following Lemma proved in \cite{AGS}.

\begin{lemma}[Lemma 3.1, \cite{AGS}]\label{LemmaAGS}
Let $\Omega$ be a non-empty, bounded, open, and convex set of $\R^n$ and let 
\begin{equation*}
    \beta(\Omega)=\frac{P(\Omega)r(\Omega)}{\abs{\Omega}}-1
        .
\end{equation*}
 Then
    \begin{equation}
        \mu\left(-\frac{\abs{\Omega}}{P(\Omega)}\right)\ge  \frac{\abs{\Omega}}{6n} \beta(\Omega).
    \end{equation}  
\end{lemma}

We are now ready to prove our last main result. 
\begin{proof}[Proof of Theorem \ref{teorema:ags}]
    From \eqref{step1}, sending $\varepsilon \to 0$, we have
    \begin{equation}\label{inizio}
    \begin{split}
     \left(\frac{T_k(\Omega)}{k}\right)^\frac{1}{k}&\geq \frac{1}{(nW_k(\Omega))^\frac{1}{k}}\int_{-r(\Omega)}^0\mu(t)^{1+\frac{1}{k}}\, dt\\
        &=\frac{1}{(nW_k(\Omega))^\frac{1}{k}}\int_{-r(\Omega)}^{-\frac{\abs{\Omega}}{P(\Omega)}}\mu(t)^{1+\frac{1}{k}}\, dt+\frac{1}{(nW_k(\Omega))^\frac{1}{k}}\int_{-\frac{\abs{\Omega}}{P(\Omega)}}^0\mu(t)^{1+\frac{1}{k}}\, dt.
    \end{split}
    \end{equation}
    Since
    \begin{equation*}
        \mu(t)=\abs{\Omega}-\int_t^0\mu'(s)\,ds\geq \abs{\Omega}+P(\Omega)t,
    \end{equation*}
    we obtain that
    \begin{equation}\label{secondoTerm}
        \begin{split}
            \frac{1}{(nW_k(\Omega))^\frac{1}{k}}\int_{-\frac{\abs{\Omega}}{P(\Omega)}}^0\mu(t)^{1+\frac{1}{k}}\,dt&\geq\frac{1}{\left(nW_k(\Omega) \right)^\frac{1}{k}}\int_{-\frac{\abs{\Omega}}{P(\Omega)}}^0(\abs{\Omega}+P(\Omega)t)^{1+\frac{1}{k}}\,dt\\
            &=\frac{1}{\left(nW_k(\Omega) \right)^\frac{1}{k}P(\Omega)}\int_{-\frac{\abs{\Omega}}{P(\Omega)}}^0(\abs{\Omega}+P(\Omega)t)^{1+\frac{1}{k}}P(\Omega)\,dt\\
            &=\frac{k}{2k+1}\frac{\abs{\Omega}^{2+\frac{1}{k}}}{\left(nW_k(\Omega) \right)^\frac{1}{k}P(\Omega)}.
        \end{split}
    \end{equation}
    On the other hand, we can bound from below the first term in \eqref{inizio} as follows
    \begin{equation}\label{primoTerm}
        \begin{split}
            \frac{1}{(nW_k(\Omega))^\frac{1}{k}}\int_{-r(\Omega)}^{-\frac{\abs{\Omega}}{P(\Omega)}}\mu(t)^{1+\frac{1}{k}}\,dt&\geq \frac{1}{P(\Omega)(nW_k(\Omega))^\frac{1}{k}} \int_{-r(\Omega)}^{-\frac{\abs{\Omega}}{P(\Omega)}}\mu(t)^{1+\frac{1}{k}}\mu'(t)\,dt\\
            &=\frac{k}{2k+1}\frac{1}{P(\Omega)(nW_k(\Omega))^\frac{1}{k}}\mu\left(-\frac{\abs{\Omega}}{P(\Omega)}\right)^{2+\frac{1}{k}}.
        \end{split}
    \end{equation}
    Then, from \eqref{inizio}-\eqref{secondoTerm}-\eqref{primoTerm}, applying Lemma \ref{LemmaAGS} and $(a+b)^k\geq a^k+ka^{k-1}b$ we find
    \begin{equation*}
        T_k(\Omega)\geq \frac{k^{k+1}}{(2k+1)^k}\frac{\abs{\Omega}^{2k+1}}{nW_k(\Omega)P(\Omega)^k}+\frac{k^{k+2}}{(2k+1)^k}\frac{\abs{\Omega}^{2k+1}}{nW_k(\Omega)P(\Omega)^k}\left(\frac{\beta(\Omega)}{6n} \right)^\frac{2k+1}{k}.
    \end{equation*}

\end{proof}

\section{P\'olya eigenvalue functional}\label{Sec4}
For the P\'olya eigenvalue functional, things get more complicated. Indeed, the so-called web-functions technique fails and just partial result can be obtained. On the other hand, in \cite{DellaPietra2012UpperBF} the authors prove an upper bound by considering as test function the distance to the boundary, obtaining
\begin{equation}
    \label{funz:dpg}\frac{\lambda_k(\Omega)\abs{\Omega}^{k+2}}{P(\Omega)^{k+1}W_{k-1}(\Omega)}\leq \frac{n(k+2)}{n-k+1}.
\end{equation}

Anyway, this functional is different from the one that we are considering $(ii)$, as we will observe in Remark \eqref{ConfrontoGDP}.

By combining the Torsion estimate \eqref{lower:torsion} with a bound for the following functional
\begin{equation}\label{Fk}
    \mathcal{G}_k(\Omega)=\frac{\lambda_k(\Omega)T_k(\Omega)}{\abs{\Omega}^{k}}
\end{equation}
obtained in Theorem \ref{PolyaFunctional},
we are able to prove a rough estimate on the functional $(ii)$. We start by proving the following bound for \eqref{Fk}.
\begin{teorema}\label{PolyaFunctional}
    Let $\Omega\subset\R^n$ be an open, bounded, and convex set of class $C^2$ and let $k=1,\dots,n$. Then, 
    \begin{equation}\label{VDB}
        \mathcal{G}_k(\Omega)\leq\binom{n}{k}.
    \end{equation}
\end{teorema}
\begin{proof}
    Let $u_\Omega$ be the solution to the Torsion problem \eqref{torssk}. Applying H\"older inequality, we have
    \begin{equation*}
        \|u_\Omega\|_{L^1(\Omega)}\leq \|u_\Omega\|_{L^{k+1}(\Omega)}\abs{\Omega}^{1-\frac{1}{k+1}},
    \end{equation*}
    that is
    \begin{equation*}
        \frac{1}{\|u_\Omega\|_{L^{k+1}(\Omega)}^{k+1}}\leq \frac{\abs{\Omega}^k}{\|u_\Omega\|_{L^1(\Omega)}^{k+1}}.
    \end{equation*}
    
    Then, testing the definition of eigenvalue in \eqref{varcar} with $u_\Omega$ and applying the previous inequality, we have
    \begin{equation*}
        \lambda_k(\Omega)\leq\frac{\int_\Omega -u_\Omega S_k(D^2 u_\Omega)\,dx}{\int_\Omega (-u_\Omega)^{k+1}\,dx}=\binom{n}{k}\frac{\|u_\Omega\|_{L^1(\Omega)}}{\|u_\Omega\|_{L^{k+1}(\Omega)}^{k+1}}\leq\binom{n}{k} \frac{\abs{\Omega}^k}{\|u_\Omega\|_{L^1(\Omega)}^k}=\binom{n}{k}\frac{\abs{\Omega}^k}{T_k(\Omega)},
    \end{equation*}
    and then
    \begin{equation*}
        \mathcal{G}_k(\Omega)\leq \binom{n}{k}.
    \end{equation*}
\end{proof}
\begin{oss}
    For $k=1$, Theorem \ref{PolyaFunctional} reduces to the estimate proved in \cite{MR43486} in the class of sets with finite measure. Subsequently, it was shown in \cite{MR3578040} that this estimate is sharp via a homogenization procedure. In the class of convex sets, the conjecture is 
    \begin{equation*}
        \frac{\lambda_1(\Omega)T_1(\Omega)}{\abs{\Omega}}\leq\frac{\pi^2}{12},
    \end{equation*}
    see \cite{MR3467375, MR3578040} for more details.
\end{oss}

\begin{proof}[Proof of Theorem \ref{TeoAut}]
    Combining \eqref{lower:torsion} and \eqref{VDB}, the proof follows.
\end{proof}

\begin{oss}\label{ConfrontoGDP}
    We can compare the functional $(ii)$ with the functional that appears in \eqref{funz:dpg} thanks to

    \begin{equation}\label{comparison Wk}
        \frac{W_{k-1}(\Omega)}{W_k(\Omega)}\geq \frac{\abs{\Omega}}{W_1(\Omega)}.
    \end{equation}
    This inequality is a consequence of the following useful formula (see \cite[Volume A, Chapter 1.2, Section 6]{HCG})
    \begin{equation*}
        W_{i+1}(\Omega)W_{i-1}(\Omega)\leq W_i^2(\Omega), \quad i=1,\dots,n-1,
    \end{equation*}
    as one can proceed in the following way 
    \begin{equation*}
        \begin{split}
            &W_1^2(\Omega)\geq \abs{\Omega}W_2(\Omega),\\
            &W_1^2(\Omega)W_2(\Omega)\geq \abs{\Omega}W_2^2(\Omega)\geq \abs{\Omega}W_1(\Omega)W_3(\Omega),\\
            &W_1^2(\Omega)W_2(\Omega)W_3(\Omega)\geq \abs{\Omega}W_1(\Omega)W_3^2(\Omega)\geq \abs{\Omega}W_1(\Omega)W_2(\Omega)W_4(\Omega).
        \end{split}
    \end{equation*}
    and, iterating this formula, we find
    \begin{equation*}
        W_1^2(\Omega)W_2(\Omega)\dots W_{k-1}(\Omega)\geq \abs{\Omega}\dots W_{k-3}(\Omega)W_{k-1}^2(\Omega)\geq \abs{\Omega}\dots W_{k-2}(\Omega)W_k(\Omega)
    \end{equation*}
    and then
    \begin{equation*}
        W_1(\Omega)W_{k-1}(\Omega)\geq \abs{\Omega}W_k(\Omega),
    \end{equation*}
    that is \eqref{comparison Wk}.
    Finally, thanks to \eqref{comparison Wk}, we find that
    \begin{equation*}
        \frac{\lambda_k(\Omega)\abs{\Omega}^{k+2}}{P(\Omega)^{k+1}W_{k-1}(\Omega)}\leq \frac{\lambda_k(\Omega)\abs{\Omega}^{k+1}}{nP(\Omega)^kW_k(\Omega)}.
    \end{equation*}
    Hence, Theorem \ref{TeoAut} gives an upper bound also for the functional in \eqref{funz:dpg}.
\end{oss}
\begin{oss}
    For $k=1$, both functionals reduce to \eqref{polyaeigenvalue}. 
\end{oss}
We remark that the Example \ref{esempio} also shows that the P\'olya eigenvalue functional is not, in general, strictly positive. Indeed, since as $\varepsilon\rightarrow0^+$ we are transforming the ellipse by an affine transformation volume preserving, the eigenvalue remains unchanged while the perimeter at the denominator is diverging.
\bibliographystyle{plain}
\bibliography{biblio}

\Addresses
\end{document}